\documentclass{siamltex}%
\usepackage{amsfonts}
\usepackage{amsmath}
\usepackage{amssymb}
\usepackage{graphicx}
\usepackage{colortbl}
\usepackage{hyperref}%
\newtheorem{remark}[theorem]{Remark}

\begin{document}

\title{Supports of Invariant Measures for Piecewise Deterministic Markov Processes }
\author{M. Bena\"{\i}m\\Institut de Math\'{e}matiques, Universit\'{e} de Neuch\^{a}tel, Neuch\^{a}tel, Suisse
\and F. Colonius and R. Lettau\\Institut f\"{u}r Mathematik, Universit\"{a}t Augsburg, Augsburg, Germany}
\maketitle

\begin{center}
\today

\end{center}

\bigskip\textbf{Abstract. }For a class of piecewise deterministic Markov
processes, the supports of the invariant measures are characterized. This is
based on the analysis of controllability properties of an associated
deterministic control system. Its invariant control sets determine the supports.

\begin{center}

\end{center}

\section{Introduction\label{Section1}}

In this paper we determine the supports of invariant measures for certain
Piecewise Deterministic Markov Processes (PDMP) using controllability
properties of an associated deterministic control system. We refer to the
monograph Davis \cite{Davis} for background on PDMP. The results extend some
of those given by Bakhtin and Hurth \cite{BH12} and Bena\"{\i}m, Le Borgne,
Malrieu and Zitt \cite{BBMZ15}, where the ergodic case is treated. We will
show that, under appropriate assumptions, the supports of the invariant
measures are determined by the invariant control sets. In particular, in the
ergodic case this reduces to one of the main results in \cite{BBMZ15} (in
particular Proposition 3.17).

A technical difference to the papers mentioned above is that, on the
deterministic side, we use control systems instead of differential inclusions.
This is due to the fact, that we bring to bear the theory of control sets
(maximal sets of complete approximate controllability) for control systems
(cf. Colonius and Kliemann \cite{ColK00}). This allows us to develop many
results in analogy to the theory for degenerate Markov diffusion processes
(cf. Arnold and Kliemann \cite{ArnoK83, ArnoK87}, Kliemann \cite{Klie87},
Colonius, Gayer, Kliemann \cite{ColGK08}) and for certain random
diffeomorphisms (Colonius, Homburg, Kliemann \cite{CHK10}).

The contents of this paper is as follows: In Section \ref{Section2} we recall
and partially strengthen some results on invariant control sets. Section
\ref{Section3} clarifies the relations between PDMP and the associated control
system, and Section \ref{Section4} establishes the relation between the
supports of invariant measures for PDMP and invariant control sets.

\textbf{Notation.} For a subset $V\subset\mathbb{R}^{d}$ the convex hull is
denoted by $\mathrm{co}(V)$. The topological closure and the interior of $V$
are denoted by $\mathrm{cl}V$ and $\mathrm{int}V$, respectively. For subsets
$A\subset V$, the closure relative to $V$ is denoted by $\mathrm{cl}_{V}A$.
Since all considered measures are probability measures, we just speak of
measures. We write $L^{\infty}(\mathbb{R}_{+},V)$ for the set of $v\in
L^{\infty}(\mathbb{R}_{+},\mathbb{R}^{d})$ such that $v(t)\in V$ for all
$t\geq0.$

\section{Controllability properties\label{Section2}}

In this section we associate deterministic control systems to PDMP and discuss
their controllability properties.

As in Bena\"{\i}m, Le Borgne, Malrieu and Zitt \cite{BBMZ15} we consider PDMP
of the following form: Let $E$ be a finite set with cardinality $m+1=\#E$, say
$E:=\{0,1,\ldots,m\}$, and for any $i\in E$ let $F^{i}:\mathbb{R}%
^{d}\rightarrow\mathbb{R}^{d}$ be a smooth ($C^{\infty}$) vector field. We
assume throughout that each $F^{i}$ is bounded, hence the (semi-)flow given by
the corresponding solution map $\Phi^{i}(t,x),t\geq0$, exists globally.
Frequently, we will also suppose that there exists a compact set
$M\subset\mathbb{R}^{d}$ that is positively invariant under each $\Phi^{i}$,
meaning that $\Phi^{i}(t,x)\in M$ for all $t\geq0$ and all $x\in M$.

We will consider a continuous-time piecewise deterministic Markov process
$Z_{t}=(X_{t},Y_{t})$ living on $\mathbb{R}^{d}\times E$. This process will be
described explicitly below, here we only remark that the continuous component
$X_{t}$ evolves according to the flows $\Phi^{i}$; the component on $E$
determines which of the flows $\Phi^{i}$ is active (with random switching
times). Already here it is clear that, in a natural way, one may associate the
following deterministic control system to the PDMP,%
\begin{equation}
\dot{x}(t)=\sum_{i=0}^{m}v_{i}(t)F^{i}(x), \label{CS}%
\end{equation}
where the (piecewise constant or measurable) control functions $v$ lie in
$L^{\infty}(\mathbb{R}_{+},S)$ and
\begin{equation}
S:=\left\{  v=(v_{i})\in\mathbb{R}^{m+1}\left\vert \sum\nolimits_{i=0}%
^{m}v_{i}=1\text{ and }v_{i}\in\{0,1\}\text{ for all }i\right.  \right\}
\label{S}%
\end{equation}
stands for the canonical basis of $\mathbb{R}^{m+1}.$ Thus only one vector
field $F^{i}$ is active at any time.

We also consider system (\ref{CS}) with convexified right hand side where the
control range is the unit $m$-simplex
\begin{equation}
\mathrm{co}(S):=\left\{  v\in\mathbb{R}^{m+1}\left\vert \sum\nolimits_{i=0}%
^{m}v_{i}=1\text{ and }v_{i}\in\lbrack0,1]\right.  \right\}  . \label{coE}%
\end{equation}
The corresponding set of control functions is $L^{\infty}(\mathbb{R}%
_{+},\mathrm{co}(S)).$

\begin{remark}
We may write $v_{0}(t)=1-v_{1}(t)-\cdots-v_{m}(t)$. Then system (\ref{CS})
with control range $S$ is equivalent to the following control system%
\begin{equation}
\dot{x}=F^{0}(x)+\sum_{i=1}^{m}v_{i}(t)\left[  F^{i}(x)-F^{0}(x)\right]
\label{CS2}%
\end{equation}
with controls having range in%
\begin{equation}
\left\{  v\in\mathbb{R}^{m}\left\vert \sum\nolimits_{i=1}^{m}v_{i}\leq1\text{
and }v_{i}\in\{0,1\}\text{ for all }i\right.  \right\}  . \label{U}%
\end{equation}
Analogously, (\ref{CS}) with control range $\mathrm{co}(S)$ is equivalent to
(\ref{CS2}) with convex control range%
\begin{equation}
\left\{  v\in\mathbb{R}^{m}\left\vert \sum\nolimits_{i=1}^{m}v_{i}\leq1\text{
and }v_{i}\in\lbrack0,1]\text{ for all }i\right.  \right\}  . \label{coU}%
\end{equation}

\end{remark}

System (\ref{CS2}) (and hence (\ref{CS})) is a special case of control-affine
systems of the form%
\begin{equation}
\dot{x}=f_{0}(x)+\sum_{i=1}^{m}v_{i}(t)f_{i}(x),(v_{i})\in\mathcal{V}%
:=L^{\infty}(\mathbb{R}_{+},V) \label{controlaffine}%
\end{equation}
with Lipschitz continuous vector fields $f_{i}$ on $\mathbb{R}^{d}$ and
compact control range $V\subset\mathbb{R}^{m}$. Next we discuss some
properties of the general class of systems of the form (\ref{controlaffine}).
We assume that (unique) global solutions $\varphi(t,x,v),t\geq0$, exist for
controls $v$ and initial condition $\varphi(0,x,v)=x$. This certainly holds in
a compact set $M$ which is positively invariant, i.e., satisfying%
\[
\varphi(t,x,v)\in M\text{ for all }x\in M,v\in\mathcal{V}\text{ and }t\geq0.
\]
We will call a subset $M$ of $\mathbb{R}^{d}$ invariant if%
\[
\{\varphi(t,x,v)\left\vert x\in M\text{ and }v\in\mathcal{V}\right.
\}=M\text{ for all }t\geq0.
\]
Define for $x\in\mathbb{R}^{d}$ and $T>0$ the reachable and controllable sets
of (\ref{controlaffine}) up to time $T$ by%
\begin{align}
\mathcal{O}_{\leq T}^{+}(x)  &  :=\{\varphi(t,x,v)\left\vert t\in
\lbrack0,T]\text{ and }v\in\mathcal{V}\right.  \},\nonumber\\
\mathcal{O}_{\leq T}^{-}(x)  &  :=\{y\left\vert x=\varphi(t,y,v)\text{ for
some }t\in\lbrack0,T]\text{ and }v\in\mathcal{V}\right.  \}, \label{orbits}%
\end{align}
and the reachable and controllable sets by
\[
\mathcal{O}^{+}(x):=\bigcup\nolimits_{T>0}\mathcal{O}_{\leq T}^{+}%
(x),\mathcal{O}^{-}(x):=\bigcup\nolimits_{T>0}\mathcal{O}_{\leq T}^{-}(x).
\]
Similarly let $\mathcal{O}_{pc}^{+}(x)$ denote the subset of $\mathcal{O}%
^{+}(x)$ which can be reached by piecewise constant control functions (i.e.,
having only finitely many discontinuities on every bounded interval), and
analogously for the other notions introduced above.

We note the following standard properties of control systems.

\begin{theorem}
\label{Theorem_approx}Consider a control system of the form
(\ref{controlaffine}). Then the following holds:

(i) For every trajectory $\varphi(t,x,v),t\in\lbrack0,T],$ of
(\ref{controlaffine}) there exists a sequence $(v_{n})$ of piecewise constant
controls with $\varphi(t,x,v_{n})\rightarrow\varphi(t,x,v)$ uniformly for
$t\in\lbrack0,T]$.

(ii) For every trajectory $\varphi(t,x,v),t\in\lbrack0,T],$ of
(\ref{controlaffine}) with control values in $\mathrm{co}(V)$ there exists a
sequence $(v_{n})$ of controls with values in $V$ such that $\varphi
(t,x,v_{n})\rightarrow\varphi(t,x,v)$, uniformly for $t\in\lbrack0,T]$.

(iii) The trajectories $\varphi(t,x,v)$ of (\ref{controlaffine}) with control
values in $\mathrm{co}(V)$ coincide with the (absolutely continuous) solutions
of the differential inclusion%
\begin{equation}
\dot{x}\in\left\{  f_{0}(x)+\sum\nolimits_{i=1}^{m}v_{i}f_{i}(x)\left\vert
(v_{i})\in\mathrm{co}(V)\right.  \right\}  . \label{DI}%
\end{equation}

\end{theorem}

\begin{proof}
For assertion (i) see Sontag \cite[Lemma 2.8.2]{Sont98}. For assertion (ii)
see, e.g., Berkovitz and Medhin \cite[Theorem IV.2.6]{BerkM12}. In assertion
(iii) it is clear, that every trajectory of (\ref{controlaffine}) with control
values in $\mathrm{co}(V)$ is a solution of the differential inclusion above,
which has compact convex velocity sets depending continuously (in the
Hausdorff metric) on $x$. The converse follows by a measurable selection
theorem, cp., e.g., Aubin and Frankowska \cite[Theorem 8.1.3]{AubF90a}.
\end{proof}

\begin{remark}
A consequence of this theorem is that the points in the reachable and
controllable sets defined in (\ref{orbits}) can be approximated using only
piecewise constant controls.
\end{remark}

We proceed to define maximal subsets of complete approximate controllability
(for some background see Colonius and Kliemann \cite{ColK00} and Kawan
\cite{Kawa13}).

\begin{definition}
A nonempty set $D\subset\mathbb{R}^{d}$ is a control set of system
(\ref{controlaffine}) if (i) it is controlled invariant, i.e., for every $x\in
D$ there is $v\in\mathcal{V}$ with $\varphi(t,x,v)\in D$ for all $t\geq0$ (ii)
for every $x\in D$ one has $D\subset\mathrm{cl}\mathcal{O}^{+}(x)$ and (iii)
$D$ is maximal with these properties. A control set $C$ is called an invariant
control set if $\mathrm{cl}C=\mathrm{cl}\mathcal{O}^{+}(x)$ for all $x\in C$.
\end{definition}

Invariant control sets need not be closed, as seen by the simple example%
\begin{equation}
\dot{x}=x(1-x)v(t),v(t)\in\lbrack-1,1]. \label{Example1}%
\end{equation}
Here $x=0$ and $x=1$ are fixed points for every $v\in\lbrack-1,1]$. Thus the
sets $\{0\},\{1\}$ and also the open interval $C:=(0,1)$ are invariant control sets.

\begin{proposition}
\label{proposition_int}(i) Let $M$ be a positively invariant subset of
$\mathbb{R}^{d}$. An invariant control set $C\subset M$ is closed relative to
$M$, i.e., $\mathrm{cl}_{M}C=C$, if for every $x\in\partial C\cap M$ the set
$\mathcal{O}^{+}(x)$ has nonvoid interior. In particular, an invariant control
set $C$ is closed if for every $x\in\partial C$ the set $\mathcal{O}^{+}(x)$
has nonvoid interior.

(ii) The compact invariant control sets coincide with the minimal compact
invariant sets, i.e., the compact invariant subsets $M\subset\mathbb{R}^{d}$
which do not contain a proper compact invariant subset.
\end{proposition}

\begin{proof}
(i) We use repeatedly that on every bounded time interval the solution of a
differential equation depends continuously on the initial value. One finds
that $\mathrm{cl}\mathcal{O}^{+}(x)\subset\mathrm{cl}C$ for all $x\in
\mathrm{cl}C\cap M$. In particular, for $x\in\partial C\cap M$ with
$\mathrm{int}\mathcal{O}^{+}(x)\not =\emptyset$ it follows that $\mathrm{int}%
\mathcal{O}^{+}(x)\subset\mathrm{cl}C\cap M$, hence there is $y\in
\mathcal{O}^{+}(x)\cap C$. Then every $z\in C$ is in $\mathrm{cl}%
\mathcal{O}^{+}(x)$. By the choice of $x$ one also has $x\in\mathrm{cl}%
\mathcal{O}^{+}(z)$ for all $z\in C$. Then one sees that $C\cup\{x\}$
satisfies the properties (i) and (ii) of a control set, hence the maximality
property of $C$ implies that $x\in C$.

(ii) Let $C$ be a compact invariant control set. Then $\mathrm{cl}\left(
\bigcup_{x\in C}\mathcal{O}^{+}(x)\right)  \subset C$, and hence Lamb,
Rasmussen and Rodrigues \cite[Proposition 3.2]{LamRR15} implies that $C$
contains a minimal compact invariant set $M$. We want to show now that $M=C$.
If $M\subsetneq C$, then, since both sets are compact, there exists a $y\in
C\setminus M$ with $\operatorname{dist}(y,M)>0$. Since for all $x\in M$, one
has $C=\mathrm{cl}\mathcal{O}^{+}(x)$, in particular $y\in\mathrm{cl}%
\mathcal{O}^{+}(x)$, the set $M$ is not invariant, which is a contradiction,
so $C$ is a minimal compact invariant set. Conversely, every minimal compact
invariant subset $M$ satisfies $\mathrm{cl}\mathcal{O}^{+}(x)\subset M$ for
all $x\in M$. By continuous dependence on the initial value, the set
$\mathrm{cl}\mathcal{O}^{+}(x)$ is positively invariant, hence it contains a
minimal compact invariant subset which by minimality of $M$ coincides with $M$.
\end{proof}

\begin{remark}
The result from \cite[Proposition 3.2]{LamRR15} used above is formulated for
set valued dynamical systems. For control systems it means that in a compact
positively invariant set $M$ there is a compact subset $N$ with
\[
\{\varphi(t,x,v)\left\vert t\geq0,x\in N,v\in\mathcal{V}\right.
\}\allowbreak=N,
\]
containing no proper subset with this property. This is proved as follows:
Consider the collection%
\[
\mathcal{K}:=\{A\subset M\left\vert A\text{ is compact with }\varphi(t,x,v)\in
A\text{ for all }t\geq0,x\in A,v\in\mathcal{V}\right.  \}.
\]
This collection is partially ordered by set inclusion and every totally
ordered subcollection has a lower bound in $\mathcal{K}$ given by the
intersection of its elements. Then Zorn's Lemma implies that there exists at
least one minimal element in $\mathcal{K}$ which turns out to be a
minimal\ compact invariant set.
\end{remark}

Control system (\ref{controlaffine}) is called locally accessible in
$x\in\mathbb{R}^{d}$ if for all $T>0$ and all neighborhoods $N$ of $x$
\begin{equation}
\mathrm{int}\mathcal{O}_{\leq T}^{+}(x)\cap N\not =\emptyset\text{ and
}\mathrm{int}\mathcal{O}_{\leq T}^{-}(x)\cap N\not =\emptyset. \label{loc_acc}%
\end{equation}
It is called locally accessible on a subset $M\subset\mathbb{R}^{d}$ (or $M$
is locally accessible) if it is locally accessible in every point $x\in M$.
Recall that the Lie algebra $\mathcal{LA}(\mathcal{F})$ generated by a family
$\mathcal{F}$ of vector fields is the smallest vector space containing
$\mathcal{F}$ that is closed under Lie brackets
\[
\lbrack f,g]:=\frac{\partial g}{\partial x}f-\frac{\partial f}{\partial x}g.
\]
The analysis of controllability properties is simplified in the following situation.

\begin{theorem}
\label{Theorem_LARC}Consider a control system of the form (\ref{controlaffine}%
) and suppose that the Lie algebra $\mathcal{LA}=\mathcal{LA}\{f_{0}%
+\sum\nolimits_{i=1}^{m}v_{i}f_{i}\left\vert v\in V\right.  \}$ satisfies for
some $x\in\mathbb{R}^{d}$%
\begin{equation}
\{g(x)\left\vert g\in\mathcal{LA}\right.  \}=\mathbb{R}^{d}\text{.}
\label{LARC}%
\end{equation}
Then the system is locally accessible in $x$ and the condition in
(\ref{loc_acc}) even holds for the reachable and controllable sets
$\mathcal{O}_{pc,\leq T}^{\pm}(x)$ corresponding to piecewise constant controls.
\end{theorem}

We note that condition (\ref{LARC}) is also necessary for local accessibility,
if the involved vector fields are real analytic; cf. Sontag \cite[Theorems 9
and 12]{Sont98} for a proof of Theorem \ref{Theorem_LARC} and the necessity statement.

\begin{remark}
For control system (\ref{CS}) (or (\ref{CS2}) with control range (\ref{U}) or
(\ref{coU})) the Lie algebra $\mathcal{LA}$ from Theorem \ref{Theorem_LARC}
coincides with the Lie algebra generated by the vector fields $F^{0}%
,\ldots,F^{m}$.
\end{remark}

In order to derive some further properties of control sets we adapt the
following lemma from Colonius and Kliemann \cite[Lemma 4.5.4]{ColK00}.

\begin{lemma}
\label{Lemma4.8}Let $x\in\mathbb{R}^{d}$ and $v\in\mathcal{V}$ with
$\varphi(T,x,v)\in\mathrm{int}\mathcal{O}_{\leq T+S}^{+}(x)$ for some $T,S>0$
and assume that the system is locally accessible at $\varphi(T,x,v)$. Then%
\[
x\in\mathrm{int}\mathcal{O}_{\leq T+2S}^{-}(\varphi(T,x,v)).
\]

\end{lemma}

\begin{proof}
We find an open neighborhood $N(y)\subset\mathrm{int}\mathcal{O}_{\leq
T+S}^{+}(x)$ of $y:=\varphi(T,x,v)$. Local accessibility at $y$ implies that
there is $z\in N(y)\cap\mathrm{int}\mathcal{O}_{\leq t_{0}}^{-}(y)$ for every
$t_{0}\in(0,S]$ small enough. Then there are a control $v$ and a neighborhood
$N(x)$ of $x$ such that $N(x)$ is mapped in a time $T_{1}\leq T+S$ via the
solution map corresponding to $v$ onto a neighborhood $N(z)$ of $z$ contained
in $N(y)\cap\mathcal{O}_{\leq t_{0}}^{-}(y)$. We obtain
\[
x\in N(x)\subset\mathcal{O}_{\leq T_{1}+t_{0}}^{-}(y)\subset\mathcal{O}_{\leq
T+2S}^{-}(\varphi(T,x,v)).
\]

\end{proof}

\begin{theorem}
\label{basic1}Consider system (\ref{controlaffine}) and assume that it is
locally accessible on a positively invariant subset $M\subset\mathbb{R}^{d}$.

(i) There are at most countably many invariant control sets $C_{r},r\in
I\subset\mathbb{N}$ in $M$. They are closed relative to $M$ and have nonvoid
interiors. The invariant control sets $C$ and their interiors are positively
invariant. Furthermore, $\mathrm{cl}_{M}(\mathrm{int}C)=C$ and for all $x\in
C$ one has $\mathrm{int}C\subset\mathcal{O}^{+}(x)$ and $C=\mathrm{cl}%
_{M}\mathcal{O}^{+}(x)$. If the Lie algebra rank condition (\ref{LARC}) holds
on $\mathrm{int}C$, then it even follows that $\mathrm{int}C\subset
\mathcal{O}_{pc}^{+}(x)$ for all $x\in C$.

(ii) Let $K\subset M$ be compact. Then there are at most finitely many
invariant control sets $C$ such that $C\cap K\not =\varnothing$.
\end{theorem}

\begin{proof}
(i) Closedness relative to $M$ follows by Proposition \ref{proposition_int}%
(i). For every invariant control set $C$, the interior of $C$ is nonvoid,
since $\mathcal{O}^{+}(x)\subset C,x\in C$, has nonvoid interior. This implies
that there are at most countably many invariant control sets, since the
topology of $\mathbb{R}^{d}$ has a countable base.

Furthermore, it is clear that $\mathrm{cl}_{M}(\mathrm{int}C)\subset C$. In
order to see the converse inclusion, consider a nonvoid open subset
$\mathcal{U}\subset\mathrm{int}C$. Then for every $t>0$ and $v\in\mathcal{V}$
the set $\{\varphi(t,x,v)\left\vert x\in\mathcal{U}\right.  \}$ is open, and
the assertion follows, since $\mathcal{O}^{+}(x)$ is dense in $C$. Note that
this argument does not use local accessibility, hence it shows that for any
invariant control set either the interior is void or dense.

Let $x\in C$ and $y\in\mathrm{int}C$. Then there is $T>0$ with $\mathrm{int}%
\mathcal{O}_{\leq T}^{-}(y)\subset C$ and hence one can steer the point $x$ to
some point $z\in\mathrm{int}\mathcal{O}_{\leq T}^{-}(y)$. Concatenating the
corresponding control with a control steering $z$ to $y$ one finds that
$y\in\mathcal{O}^{+}(x)$ implying $\mathrm{int}C\subset\mathcal{O}^{+}(x)$ and
also $C\subset\mathrm{cl}_{M}\mathcal{O}^{+}(x)$ for all $x\in C$. If
(\ref{LARC}) holds on $C$, Theorem \ref{Theorem_LARC} allows us to replace
$\mathrm{int}\mathcal{O}_{\leq T}^{-}(y)$ by $\mathrm{int}\mathcal{O}_{pc,\leq
T}^{-}(y)$ in this argument. By the approximation property in Theorem
\ref{Theorem_approx}(i) one can steer the point $x$ into $\mathrm{int}%
\mathcal{O}_{pc,\leq T}^{-}(y)$. Hence $\mathrm{int}C\subset\mathcal{O}%
_{pc}^{+}(x)$ for all $x\in C$.

In order to show positive invariance of $\mathrm{int}C$ suppose that there are
$x\in\mathrm{int}C,T_{0}>0$ and $v_{0}\in\mathcal{V}$ with $\varphi
(T_{0},x,v_{0})\not \in \mathrm{int}C$. Let%
\[
T_{1}:=\sup\{t\in\lbrack0,T_{0}]\left\vert \varphi(s,x,v_{0})\in
\mathrm{int}C\text{ for all }s\in\lbrack0,t]\right.  \}.
\]
Then it follows that $y:=\varphi(T_{1},x_{0},v_{0})\in(C\cap M)\setminus
(\mathrm{int}C)\subset\mathrm{cl}_{M}C=C$. Hence there are $\tau>0$ and
$v_{1}\in\mathcal{V}$ with $\varphi(\tau,y,v_{1})\in\mathrm{int}C$. Observe
that every neighborhood of $y$ intersects $\{\varphi(t,x,v_{0})\left\vert
t\in(T_{1},T_{0}]\right.  \}\cap M$. By continuous dependence on the initial
value a neighborhood of $y$ is mapped into the interior of $C$ and hence there
is $T_{2}\in(T_{1},T_{0}]$ such that for all $t\in\lbrack T_{1},T_{2})$ the
intersection $\mathcal{O}^{+}(\varphi(t,x,v_{0}))\cap\mathrm{int}C$ is
nonvoid. Then it follows that the point $x\in\mathrm{int}C$ is in
$\mathcal{O}^{+}(\varphi(t,x,v_{0}))\subset\mathcal{O}^{+}(x)$. This easily
implies that also all points $\varphi(t,x,v_{0}),t\in\lbrack T_{1},T_{2})$ are
in $\mathrm{int}C$ contradicting the definition of $T_{1}$. It follows that
$\mathrm{int}C$ is positively invariant. Similar, but simpler arguments show
that also $C$ is positively invariant.

(ii) If the assertion is false, one finds infinitely many invariant control
sets $C_{n},n\in\mathbb{N}$ and points $x_{n}\in$ $\mathrm{cl}C_{n}\cap K$.
This sequence has a cluster point and every cluster point $x$ is in $K\subset
M$. By local accessibility on $M$ there are $T,S>0$ and a control $v$ with
$\varphi(T,x,v)\in\mathrm{int}\mathcal{O}_{\leq T+S}^{+}(x)$. By Lemma
\ref{Lemma4.8} we find that there is an open neighborhood $V(x)$ of $x$ with
$V(x)\subset\mathrm{int}\mathcal{O}_{\leq T+2S}^{-}(\varphi(T,x,v))$. For
$n\in\mathbb{N}$ large enough $x_{n}\in V(x)$ and $x_{n}\in\mathrm{cl}C_{n}$,
hence one finds $y_{n}\in C_{n}$%
\[
y_{n}\in\mathcal{O}_{\leq T+2S}^{-}(\varphi(T,x,v)).
\]
This shows that for all $n$ large enough the points $y_{n}\in C_{n}$ can be
steered to the single point $\varphi(T,x,v)$. This contradicts positive
invariance of the pairwise different invariant control sets $C_{n}$.
\end{proof}

\begin{remark}
In the system given by (\ref{Example1}), the invariant control set $C=(0,1)$
is trivially closed relative to the positively invariant set $M=(0,1)$, since
$\partial C\cap M=\varnothing$ (observe that here local accessibility on $M$
holds). In the absence of local accessibility, Colonius and Kliemann
\cite[Example 3.2.9]{ColK00} presents an example of an invariant control set
which is neither open nor closed in $\mathbb{R}^{2}$ and which is not
positively invariant.
\end{remark}

The next theorem does not assume that the system is locally accessible on $M$.

\begin{theorem}
\label{basic2}Consider system (\ref{controlaffine}) on a positively invariant
subset $M\subset\mathbb{R}^{d}$.

(i) If there exists a compact subset $K\subset M$ such that for every $x\in M$
one has $\mathrm{cl}\mathcal{O}^{+}(x)\cap K\not =\emptyset$, then for every
$x\in M$ there exists an invariant control set $C\subset\mathrm{cl}%
_{M}\mathcal{O}^{+}(x)$. If the system is also locally accessible in $M$, then
there are only finitely many invariant control sets in $M$.

(ii) Conversely, if for every $x\in M$ there exists an invariant control set
$C\subset\mathrm{cl}_{M}\mathcal{O}^{+}(x)$ and there are only finitely many
invariant control sets in $M$, then there exists a compact subset $K\subset M$
such that for every $x\in M$ one has $\mathrm{cl}\mathcal{O}^{+}(x)\cap
K\not =\emptyset$.
\end{theorem}

\begin{proof}
(i) Note first that $\mathrm{cl}\mathcal{O}^{+}(x)\cap K=\mathrm{cl}%
_{M}\mathcal{O}^{+}(x)\cap K$ for every $x\in M$. In order to show that for
every $x\in M$ there is an invariant control set $C\subset\mathrm{cl}%
_{M}\mathcal{O}^{+}(x)$ define $K(y):=\mathrm{cl}\mathcal{O}^{+}(y)\cap K$ for
$y\in\mathrm{cl}_{M}\mathcal{O}^{+}(x)$. Consider the family $\mathcal{K}$ of
nonvoid and compact subsets of $M$ given by $\mathcal{K}=\left\{
K(y)\left\vert y\in\mathrm{cl}_{M}\mathcal{O}^{+}(x)\right.  \right\}  $. Then
$\mathcal{K}$ is ordered via
\[
K(z)\prec K(y)\;\text{if }y\in\mathrm{cl}_{M}\mathcal{O}^{+}(z).
\]
In fact, this is an order: If $K(z)\prec K(y)$ and $K(y)\prec K(z)$, then
$y\in\mathrm{cl}_{M}\mathcal{O}^{+}(z)$ and $z\in\mathrm{cl}_{M}%
\mathcal{O}^{+}(y)$ implying $\mathrm{cl}_{M}\mathcal{O}^{+}(y)=\mathrm{cl}%
_{M}\mathcal{O}^{+}(z)$ and hence $K(y)=K(x)$. If $K(y_{1})\prec K(y_{2})$ and
$K(y_{2})\prec K(y_{3})$, then $y_{2}\in\mathrm{cl}_{M}\mathcal{O}^{+}(y_{1})$
and $y_{3}\in\mathrm{cl}_{M}\mathcal{O}^{+}(y_{2})$, hence $y_{3}%
\in\mathrm{cl}_{M}\mathcal{O}^{+}(y_{1})$ implying $K(y_{1})\prec K(y_{3})$.

Every linearly ordered set $\left\{  K(y_{i})\left\vert i\in I\right.
\right\}  $ has an upper bound $K(y)\;$for some $y\in\bigcap_{i\in I}K(y_{i}%
)$, because the intersection of decreasing compact subsets of the compact set
$K$ is nonvoid. Thus Zorn's lemma implies that the family $\mathcal{K}$ has a
maximal element $K(y)$. Now we claim that the set
\[
C:=\mathrm{cl}_{M}\mathcal{O}^{+}(y)
\]
is an invariant control set. It is clear that $C\subset\mathrm{cl}%
_{M}\mathcal{O}^{+}(x)$. Every $z\in C$ is approximately reachable from $y$,
hence $K(y)\prec K(z)$ and maximality of $K(y)$ implies $K(y)=K(z)$, hence for
every $z\in C$ one has $y\in\mathrm{cl}_{M}\mathcal{O}^{+}(z)$. Then it
follows that $C=\mathrm{cl}_{M}\mathcal{O}^{+}(z)$ for every $z\in C$. This
implies for every $z\in C$ that there is $v\in\mathcal{V}$ with $\varphi
(t,z,v)\in C$ for all $t\geq0$ and $\mathrm{cl}C=\mathrm{cl}\mathcal{O}%
^{+}(z)$. Hence $C$ is an invariant control set.

If $M$ is locally accessible, the invariant control sets are closed in $M$ by
Theorem \ref{basic1}(i). Thus, if $x$ is in an invariant control set $C$, then
$\mathrm{cl}_{M}\mathcal{O}^{+}(x)=C$ and hence $C\cap K\not =\emptyset$. By
Theorem \ref{basic1}(ii) only finitely many invariant control sets have
nonvoid intersection with $K$, hence only finitely many invariant control sets
in $M$ exist.

(ii) For each of the finitely many invariant control sets $C_{i}$ pick
$x_{i}\in C_{i}$ and define $K:=\{x_{i}\left\vert 1\leq i\leq N\right.  \}$.
Then for each $x\in M$ there is $i\in\{1,\ldots,N\}$ with $C_{i}%
\subset\mathrm{cl}_{M}\mathcal{O}^{+}(x)$, hence $x_{i}\in\mathrm{cl}%
\mathcal{O}^{+}(x)\cap K$.
\end{proof}

The following corollary is a consequence of Theorems \ref{basic1} and
\ref{basic2}.

\begin{corollary}
\label{Corollary_inv}Suppose that $M$ is a compact positively invariant set
for system (\ref{controlaffine}).

(i) For every $x\in M$ there is an invariant control set $C\subset M$ with
$C\subset\mathrm{cl}\mathcal{O}^{+}(x)$. If $C$ is closed and $\mathrm{int}%
C\not =\emptyset$ then $C=\mathrm{cl}(\mathrm{int}C)$ and for every $x\in M$
one has $\mathrm{int}C\subset\mathcal{O}^{+}(x)$ and $\mathrm{int}%
C\cap\mathcal{O}_{pc}^{+}(x)\not =\emptyset$.

(ii) Suppose that $M$, additionally, is locally accessible. Then $M~$contains
at least one and at most finitely many invariant control sets $C_{r}%
,r=1,\ldots,l$. They are compact (hence characterized by Proposition
\ref{proposition_int} (ii)), have nonvoid interiors $\mathrm{int}C_{r}$ and
for every point $x\in M$ there is $r\in\{1,\ldots,l\}$ with $\mathrm{int}%
C\cap\mathcal{O}_{pc}^{+}(x)\not =\emptyset$. If the Lie algebra rank
condition (\ref{LARC}) holds on $\mathrm{int}C$, then it even follows that
$\mathrm{int}C\subset\mathcal{O}_{pc}^{+}(x)$ for all $x\in C$.
\end{corollary}

\begin{proof}
(i) We only have to show that for every $x\in M$ one has $\mathrm{int}%
C\subset\mathcal{O}^{+}(x)$ and $\mathrm{int}C\cap\mathcal{O}_{pc}%
^{+}(x)\not =\emptyset$, since the other assertions follow from Theorem
\ref{basic2}(i). Note that there is $y\in\mathcal{O}^{+}(x)\cap\mathrm{int}C$,
since $C\subset\mathrm{cl}\mathcal{O}^{+}(x)$. By Theorem \ref{basic1}(i), it
follows that $\mathrm{int}C\subset\mathcal{O}^{+}(y)\subset\mathcal{O}^{+}%
(x)$. Furthermore, $y\in\mathcal{O}^{+}(x)\cap\mathrm{int}C$ implies by
Theorem \ref{Theorem_approx}(i) that $\mathcal{O}_{pc}^{+}(x)\cap
\mathrm{int}C\not =\emptyset$.

(ii) This is immediate from Theorems \ref{basic1} and \ref{basic2}.
\end{proof}

\begin{remark}
Consider a linear control system of the form $\dot{x}(t)=Ax(t)+Bv(t),v(t)\in
V\subset\mathbb{R}^{m}$ with matrices $A\in\mathbb{R}^{d\times d}%
,B\in\mathbb{R}^{d\times m}$ and compact control range $V\subset\mathbb{R}%
^{m}$ with $0\in\mathrm{int}V$. Suppose that the controllability rank
condition $\mathrm{rank}[B,AB,\ldots,A^{d-1}B]=d$ holds (this condition,
called Kalman's rank condition, is equivalent to the fact that the system
without control constraint satisfies $\mathcal{O}^{+}(x)=\mathbb{R}^{d}$ for
all $x\in\mathbb{R}^{d}$). Then there exists a unique control set with nonvoid
interior. It is a compact invariant control set if $A$ is stable, i.e., all
eigenvalues of $A$ have negative real parts (cp. Colonius and Kliemann
\cite[Example 3.2.16]{ColK00}). The example in \cite[Section 5.2]{BBMZ15}
discusses for such a situation (with $d=2$) the invariant measures.
\end{remark}

Bena\"{\i}m, Le Borgne, Malrieu and Zitt in \cite{BBMZ15} define for the
control system (\ref{CS}) with control range (\ref{S}) and a positively
invariant compact set $M\subset\mathbb{R}^{d}$ the accessible set%
\begin{equation}
\Gamma:=\bigcap_{x\in M}\mathrm{cl}\mathcal{O}_{pc}^{+}(x). \label{gamma}%
\end{equation}
See also Proposition 3.11 in \cite{BBMZ15} for the relations with the
associated differential inclusion.

In complete analogy, one can also define the accessible set $\Gamma$ of a
general control system of the form (\ref{controlaffine}) and any
$M\subset\mathbb{R}^{d}.$ Its relation to invariant control sets is clarified
in the following proposition.

\begin{proposition}
\label{proposition_gamma}Consider control system (\ref{controlaffine}) and let
$M\subset\mathbb{R}^{d}$.

(i) If $\Gamma$ is a nonvoid subset of $M$, then $\Gamma$ is a closed
invariant control set.

(ii) Let $M$ be compact and positively invariant and suppose that there is
only one invariant control set $C$ in $M$. Then $C=\Gamma$, in particular, $C$
is closed.

(iii) If a compact positively invariant set $M$ contains two closed invariant
control sets, then $\Gamma$ is empty.
\end{proposition}

\begin{proof}
(i) Suppose that $\Gamma=\bigcap_{x\in M}\mathrm{cl}\mathcal{O}^{+}%
(x)\not =\emptyset$. Continuous dependence on the initial value shows that
$\Gamma$ is positively invariant, hence $\mathrm{cl}\mathcal{O}^{+}%
(x)\subset\Gamma$ for every $x\in\Gamma$. For all $x,y\in\Gamma$ one has
$y\in\mathrm{cl}\mathcal{O}^{+}(x)$ and $x\in\mathrm{cl}\mathcal{O}^{+}(y)$.
Thus it is an invariant control set.

(ii) By Corollary \ref{Corollary_inv}, for every $x\in M$ there is an
invariant control set\ in $\mathrm{cl}\mathcal{O}^{+}(x)$, hence the inclusion
$C\subset\bigcap_{x\in M}\mathrm{cl}\mathcal{O}^{+}(x)$ holds. For the
converse inclusion, note that this implies that $\Gamma$ is a nonvoid subset
of $M$, and hence the assertion follows by (i).

(iii) If $\Gamma$ is nonvoid, it is by (i) a closed invariant control set. Let
$C\not =\Gamma$ be another closed invariant control set. Then $\mathrm{cl}%
\mathcal{O}^{+}(x)\subset C$ for $x\in C$. Since, by the maximality property,
the intersection of two control sets $C_{1}\not =C_{2}$ is void, one obtains
the contradiction $\Gamma=\emptyset$.
\end{proof}

\begin{remark}
The accessible set $\Gamma$ may be nonvoid, if two invariant control sets
exist and one of them is not closed. The paper by Bena\"{\i}m and Lobry
\cite{BenL16} presents Lotka-Volterra systems where this occurs.
\end{remark}

Next we discuss the points which can be steered into an invariant control set.

\begin{definition}
The domain of attraction of an invariant control set $C$ is%
\[
\mathcal{A}(C):=\{x\in\mathbb{R}^{d}\left\vert \mathrm{cl}\mathcal{O}%
^{+}(x)\cap C\not =\emptyset\right.  \}
\]
and its strict domain of attraction is
\[
\mathcal{A}_{strict}(C):=\{x\in\mathbb{R}^{d}\left\vert \mathrm{cl}%
\mathcal{O}^{+}(x)\cap C\not =\emptyset\text{ and }\mathrm{cl}\mathcal{O}%
^{+}(x)\cap C^{\prime}\not =\emptyset\text{ implies }C^{\prime}=C\right.  \},
\]
where $C^{\prime}$\ denotes any invariant control set.
\end{definition}

It is easily seen that for an invariant control set $C$ with nonvoid interior
$x\in\mathcal{A}(C)$ if and only if $\mathcal{O}^{+}(x)\cap\mathrm{int}%
C\not =\emptyset$. In fact, for $y\in\mathrm{cl}\mathcal{O}^{+}(x)\cap C$
there are $v\in\mathcal{V}$ and $t>0$ with $\varphi(t,x,v)\in\mathrm{int}C$
and hence, by continuous dependence on the initial value, it follows that
$\mathcal{O}^{+}(x)\cap\mathrm{int}C\not =\emptyset$ (recall that for locally
accessible systems, every invariant control set has nonvoid interior.) \ This
also shows that the domain of attraction is open. The points in the strict
domain of attraction can only be steered to a single invariant control set. If
$C$ is closed, the strict domain of attraction trivially contains $C$. It need
not be open, since, e.g., a point of the boundary $C\cap\partial C$ may be in
$\partial\left(  \mathcal{A}_{strict}(C)\right)  $.

If the system is locally accessible on a compact positively invariant set
$M\subset\mathbb{R}^{d}$, Corollary \ref{Corollary_inv} implies that every
$x\in M$ is in the domain of attraction of at least one of the finitely many
invariant control set. Furthermore, if $M$ is also connected and contains at
least two invariant control sets, it follows that for every invariant control
set $C_{k}$ the set $\left[  \mathcal{A}(C_{k})\cap M\right]  \setminus
\mathcal{A}_{strict}(C_{k})$ is nonvoid. Otherwise, one would obtain a
decomposition of $M$ into two disjoint open sets%
\[
M=\left[  \mathcal{A}(C_{k})\cap M\right]  \cup%
{\textstyle\bigcup_{i\not =k}}
\left[  \mathcal{A}(C_{i})\cap M\right]  .
\]
Consequently, one finds for every invariant control set $C_{k}$ an invariant
control set $C_{i}\not =C_{k}$ with%
\[
\mathcal{A}(C_{k})\cap\mathcal{A}(C_{i})\not =\emptyset.
\]
Thus there are points which can be steered into two different invariant
control sets.

Finally, we show that one may consider an invariant control set $C$ as an
\textquotedblleft accessible set\textquotedblright\ (similar to (\ref{gamma}))
for a neighborhood of $C$, which, however, need not be positively invariant.

\begin{proposition}
\label{PropositionC_Gamma}Let $M$ be a compact positively invariant and
locally accessible set. Then for every invariant control set $C\subset M$
there is a compact neighborhood $\mathcal{U}$ of $C$ such that%
\[
C=\bigcap_{x\in\mathcal{U}}\mathrm{cl}\mathcal{O}_{pc}^{+}(x).
\]

\end{proposition}

\begin{proof}
There are only finitely many invariant control sets in $M$ and they are
compact. For every point $x\in\partial C$ there are $T>0$ and a piecewise
constant control $v$ with $\varphi(T,x,v)\in\mathrm{int}C$. Using compactness
of $\partial C$ and continuous dependence on the initial value one finds a
neighborhood $\mathcal{U}$ of $C$ such that $C\subset\bigcap_{x\in\mathcal{U}%
}\mathrm{cl}\mathcal{O}_{pc}^{+}(x)$. Using the definition of invariant
control sets one sees that here equality holds.
\end{proof}

\section{Piecewise Deterministic Markov Processes\label{Section3}}

We will now, following Bena\"{\i}m, Le Borgne, Malrieu and Zitt \cite{BBMZ15},
define piecewise deterministic processes. With the notation introduced in
Section \ref{Section2}, assume that there is a compact set $M\subset
\mathbb{R}^{d}$, that is positively invariant for every flow $\Phi^{i}$,
i.e.,\thinspace\ $\Phi_{t}^{i}(M)\subset M$ for all $t\geq0$ and all $i\in E$.

Let $x\mapsto Q(x) = (Q(x,i,j))_{ij} : \mathbb{R}^{d}\mapsto\mathbb{R}%
^{(m+1)\times(m+1)}$ be continuous with $Q(x)$ an irreducible, aperiodic
Markov transition matrix, which means that for all $x$ there is $n_{x}%
\in\mathbb{N}$ with $Q^{n_{x}}(x,i,j)>0$ for all $i,j\in E$. Let
$(N_{t})_{t\geq0}$ be a homogenous Poisson process with intensity $\lambda$,
jumps times $(T_{n})_{n\geq0}$ and denote by $(U_{n})_{n\geq1}$ with
$U_{n}=T_{n}-T_{n-1}$ the times between the jumps. Assume that $\tilde{Z}%
_{0}\in M\times E$ is a random variable independent of $(N_{t})_{t\geq0}$.

We define the discrete-time process $(\tilde{Z}_{n})_{n}=(\tilde{X}_{n}%
,\tilde{Y}_{n})_{n}$ on $M\times E$ recursively by
\begin{align*}
\tilde{X}_{n+1}  &  =\Phi^{\tilde{Y}_{n}}(U_{n+1},\tilde{X}_{n}),\\
\mathbb{P}\left[  \tilde{Y}_{n+1}=j\left\vert \tilde{X}_{n+1},\tilde{Y}%
_{n}=i\right.  \right]   &  =Q(\tilde{X}_{n+1},i,j).
\end{align*}
and by interpolation its continuous time version $(Z_{t})_{t\geq0}$
\begin{equation}
Z_{t}=\left(  \Phi^{\tilde{Y}_{n}}(t-T_{n},\tilde{X}_{n}),\tilde{Y}%
_{n}\right)  \text{ for }t\in\lbrack T_{n},T_{n+1}). \label{Z}%
\end{equation}
We define for $n\in\mathbb{N}$%
\[
\mathbb{T}_{n}=\{(\mathbf{i},\mathbf{u})=((i_{0},i_{1},\cdots i_{n}%
),(u_{1},\cdots u_{n}))\in E^{n+1}\times\mathbb{R}_{+}^{n}\}
\]
and
\[
\mathbb{T}_{n}^{ij}=\left\{  (\mathbf{i},\mathbf{u})\in\mathbb{T}_{n}%
|i_{0}=i,i_{n}=j\right\}  .
\]
Then we can define the trajectory with initial value $x\in M$ induced by
$(\mathbf{i},\mathbf{u})$: With $t_{0}=0,t_{k}=t_{k-1}+u_{k},k=1,\ldots,n$,%
\[
\eta_{x,\mathbf{i},\mathbf{u}}(t)=%
\begin{cases}
x & t=0\\
\Phi_{t-t_{k-1}}^{i_{k-1}}(x_{k-1}) & t_{k-1}<t\leq t_{k}\\
\Phi_{t-t_{n}}^{i_{n}}(x_{n}) & t>t_{n}%
\end{cases}
.
\]
We can then denote
\[
\Phi_{\mathbf{u}}^{\mathbf{i}}(x)=\eta_{x,\mathbf{i},\mathbf{u}}(t_{n})
\]
and note that $(x,\mathbf{u})\mapsto\Phi_{\mathbf{u}}^{\mathbf{i}}(x)$ is
continuous. In terms of the associated deterministic control system (\ref{CS})
this means that we consider a piecewise constant control function $v$ with
values in $S$ generating the trajectory $\varphi(t,x,v)=\eta_{x,\mathbf{i}%
,\mathbf{u}}(t)$.

We can define
\begin{equation}
p(x,\mathbf{i},\mathbf{u}):=\prod_{j=1}^{n}Q(x_{j},i_{j-1},i_{j}) \label{p}%
\end{equation}
and the set of adapted elements%
\[
\mathbb{T}_{n,ad(x)}=\left\{  (\mathbf{i},\mathbf{u})\in\mathbb{T}%
_{n}:p(x,\mathbf{i},\mathbf{u})>0\right\}  .
\]
Note that $x\mapsto p(x,\mathbf{i},\mathbf{u})$ is continuous.

The relation between the trajectories of the Piecewise Deterministic Markov
Process and control system (\ref{CS}) is clarified by the following results
from \cite{BBMZ15}, slightly reformulated for control systems instead of
differential inclusions. The first result is, in the terminology of Arnold and
Kliemann \cite{ArnoK83}, a tube lemma. It shows that tubes around any
(finite-time) trajectory of the control system have positive probability. It
reformulates the development in \cite[Section 3.1]{BBMZ15}. Recall that
$\mathrm{co}(S)$ is the unit $m$-simplex.

\begin{lemma}
[Tube lemma]\label{tubelemma} For all $T>0$, $x\in M,$ $i\in E,\delta>0$ and
every trajectory $\varphi(\cdot,x,v)$ of system (\ref{CS}) with control $v$ in
$L^{\infty}(\mathbb{R}_{+},\mathrm{co}(S))$ there is $\varepsilon
=\varepsilon(x,i,v,T,\delta)>0$ such that
\[
\mathbb{P}_{x,i}\left[  \sup_{t\in\lbrack0,T]}\Vert X_{t}-\varphi
(t,x,v)\Vert\leq\delta\right]  \geq\varepsilon.
\]

\end{lemma}

\begin{proof}
Theorem \ref{Theorem_approx}(ii) shows that for every control $v\in L^{\infty
}(\mathbb{R}_{+},\mathrm{co}(S))$ there is a control $v_{1}\in L^{\infty
}(\mathbb{R}_{+},S)$ such that
\[
\sup_{t\in\lbrack0,T]}\Vert\varphi(t,x,v)-\varphi(t,x,v_{1})\Vert\leq
\frac{\delta}{3}.
\]
By Theorem \ref{Theorem_approx}(i) we can find a piecewise constant control
with values in $S$ such that the corresponding trajectory approximates
$\varphi(t,x,v_{1})$ uniformly on $[0,T]$. Thus there are $n\in\mathbb{N}$ and
$(\mathbf{{i},{u})}\in\mathbb{T}_{n}$%
\[
\sup_{t\in\lbrack0,T]}\Vert\varphi(t,x,v_{1})-\eta_{x,\mathbf{i},\mathbf{u}%
}(t)\Vert\leq\frac{\delta}{3}.
\]
By \cite[Lemma 3.2]{BBMZ15} there is $\varepsilon>0$ such that
\[
\mathbb{P}_{x,i}\left[  \sup_{t\in\lbrack0,T]}\Vert X_{t}-\eta_{x,\mathbf{{i}%
,{u}}}(t)\Vert\leq\frac{\delta}{3}\right]  \geq\varepsilon>0.
\]
Here $\varepsilon$ depends on $x,i,\mathbf{i},\mathbf{u}$ and, naturally, on
$T$ and $\delta$. In the situation above, $(\mathbf{i},\mathbf{u})$ is
determined by $x$ and $v$, hence we may write $\varepsilon=\varepsilon
(x,i,v,T,\delta)$ \ Taken together, this implies%
\begin{align*}
&  \mathbb{P}_{x,i}\left[  \sup_{t\in\lbrack0,T]}\Vert X_{t}-\varphi
(t,x,v)\Vert\leq\delta\right] \\
&  =\mathbb{P}_{x,i}\left[  \sup_{t\in\lbrack0,T]}\left\Vert X_{t}%
-\eta_{x,\mathbf{i},\mathbf{u}}(t)+\eta_{x,\mathbf{i},\mathbf{u}}%
(t)-\varphi(t,x,v_{1})+\varphi(t,x,v_{1})-\varphi(t,x,v)\right\Vert \leq
\delta\right] \\
&  \geq\mathbb{P}_{x,i}\left[  \sup_{t\in\lbrack0,T]}\Vert X_{t}%
-\eta_{x,\mathbf{i},\mathbf{u}}(t)\Vert\leq\frac{\delta}{3}\right]
\geq\varepsilon(x,i,v,T,\delta).
\end{align*}

\end{proof}

The next result establishes a relation between the law of the continuous-time
process and the associated control system (with convexified control range).

\begin{theorem}
Consider control system (\ref{CS}) with controls in $L^{\infty}(\mathbb{R}%
_{+},\mathrm{co}(S))$ and let $x\in M$. If $X_{0}=x\in M$ then the support of
the law of $\left(  X_{t}\right)  _{t\geq0}$ equals the set of trajectories
starting in $x$ given by%
\[
\{\varphi(\cdot,x,v)\in C(\mathbb{R}_{+},\mathbb{R}^{d})\left\vert v\in
L^{\infty}(\mathbb{R}_{+},\mathrm{co}(S))\right.  \}.
\]

\end{theorem}

\begin{proof}
By \cite[Theorem 3.4]{BBMZ15}, the support of the law of $\left(
X_{t}\right)  _{t\geq0}$ equals the set of trajectories $\eta(\cdot)$ with
$\eta(0)=x$ of the differential inclusion%
\[
\dot{\eta}(t)\in\left\{  \sum_{i=0}^{m}v_{i}(t)F^{i}(x)\left\vert \sum
_{i=0}^{m}v_{i}(t)=1\text{ and }v_{i}\in\lbrack0,1]\}\right.  \right\}  .
\]
By Theorem \ref{Theorem_approx}(iii) the set of trajectories of this
differential inclusion coincides with the set of trajectories of the
associated control system.
\end{proof}

A first consequence of the tube lemma is the following.

\begin{proposition}
\label{lemma_entry} Let $C\subset M$ be an invariant control set with nonvoid
interior and $x\in\mathcal{A}(C)$. Then there are $T_{x}>0$ and $\varepsilon
_{x}>0$ with
\[
\mathbb{P}_{x,i}\left[  X_{T_{x}}\in\mathrm{int}C\right]  \geq\varepsilon
_{x}\text{ for all }i\in E.
\]

\end{proposition}

\begin{proof}
By the definition of $C$ there are $T_{x}\geq0$ and a control $v_{x}\in
L^{\infty}(\mathbb{R}_{+},\mathrm{co}(S))$ with $y:=\varphi(T_{x},x,v_{x}%
)\in\mathrm{int}C$, hence $B_{\delta_{x}}(y)\subset\mathrm{int}C$ for some
$\delta_{x}>0$. Hence the tube lemma, Lemma \ref{tubelemma}, implies
\begin{align*}
\mathbb{P}_{x,i}[X_{T_{x}}\in\mathrm{int}C]  &  \geq\mathbb{P}_{x,i}[X_{T_{x}%
}\in B_{\delta}(y)]\geq\mathbb{P}_{x,i}\left[  \Vert X_{T_{x}}-\varphi
(T_{x},x,v_{x})\Vert\leq\delta\right] \\
&  \geq\mathbb{P}_{x,i}\left[  \sup_{t\in\lbrack0,T_{x}]}\Vert X_{t}%
-\varphi(t,x,v_{x})\Vert\leq\delta\right]  \geq\varepsilon(x,i,v_{x}%
,T_{x},\delta_{x}),
\end{align*}
Since $v_{x},T_{x}$ and $\delta_{x}$ are determined by $x$ and $E$ is a finite
set we find that%
\[
\varepsilon_{x}:=\min_{i\in E}\varepsilon(x,i,v_{x},T_{x},\delta_{x})>0.
\]

\end{proof}

\section{Invariant measures and their supports\label{Section4}}

The purpose of this section is to show that the support of every invariant
measure is contained in the union of the invariant control sets times $E$ and
that, for an ergodic invariant measure, the support coincides with an
invariant control set times $E$.

Recall that a point $x_{0}$ is in the support of a measure if every
neighborhood of $x_{0}$ has positive measure and that the support is closed.

For the following results on existence of invariant measures for the discrete
time process and the continuous time process compare Bena\"{\i}m, Le Borgne,
Malrieu and Zitt \cite{BBMZ15}. First one notes that the invariant measures of
the discrete-time process and the continuous-time process are homeomorphic to
each other, the homeomorphism preserves ergodicity and the supports coincide
\cite[Proposition 2.4 and Lemma 2.6]{BBMZ15}.

In our context, \cite[Lemma 3.16]{BBMZ15} is replaced by the following
technical lemma.

\begin{lemma}
\label{lemma41}Let $M$ be a compact positively invariant set for system
(\ref{CS}) with controls in $L^{\infty}(\mathbb{R}_{+},\mathrm{co}(S))$ and
suppose that there are only finitely many invariant control sets $C_{r}\subset
M,r=1,\dots,l$, and $C_{r}=\mathrm{cl}(\mathrm{int}C_{r})$ and $\mathrm{int}%
C_{r}\subset\mathcal{O}^{+}(x)$ for every $x\in C_{r}$ and all $r$.

Consider the process $(\tilde{Z}_{n})_{n}=(\tilde{X}_{n},\tilde{Y}_{n})_{n}$.
Pick $p_{r}\in\mathrm{int}C_{r}$ and let $\mathcal{U}_{r}$ be an open
neighborhood of $p_{r}$ with $\mathrm{cl}\mathcal{U}_{r}\subset\mathrm{int}%
C_{r}$ and $\mathcal{U}:=\bigcup_{r=1}^{l}\mathcal{U}_{r}$ and, finally,
choose $i,j\in E$.

Then there exist $m\in\mathbb{N}$ and $\varepsilon,\beta>0$, finite sequences
$(\mathbf{i}^{1},\mathbf{u}^{1}),\ldots,(\mathbf{i}^{N},\mathbf{u}^{N}%
)\in\mathbb{T}_{m}^{ij}$ and an open covering $\mathcal{O}^{1},\ldots
,\mathcal{O}^{N}$ of $M$ such that for all $x\in M$ and $\mathbf{t}%
\in\mathbb{R}_{+}^{m},$%
\[
x\in\mathcal{O}^{k}\text{ and }\left\Vert \mathbf{t}-\mathbf{u}^{k}\right\Vert
<\varepsilon\text{ implies }\Phi_{\mathbf{t}}^{\mathbf{i}^{k}}(x)\in
\mathcal{U}\text{ and }p(x,\mathbf{i}^{k},\mathbf{t})\geq\beta.
\]
Furthermore, $m,\varepsilon$ and $\beta$ can be chosen independently of
$i,j\in E$.
\end{lemma}

\begin{remark}
If $M$ is a compact positively invariant and locally accessible set, then the
assumptions of Lemma \ref{lemma41} hold by Corollary \ref{Corollary_inv}.
\end{remark}

\begin{proof}
Fix $i$ and $j$ and consider open neighborhoods $\mathcal{W}_{r}$ of $p_{r}$
with $\mathrm{cl}\mathcal{W}_{r}\subset\mathcal{U}_{r},r=1,\ldots,l$. Let
$\mathcal{W}:=\bigcup_{r=1}^{l}\mathcal{W}_{r}$. For all $\beta>0$ and all
finite sequences $(\mathbf{i},\mathbf{u})$ the sets
\[
\mathcal{O}\left(  \mathbf{i},\mathbf{u},\beta\right)  :=\left\{  x\in
M\left\vert \Phi_{\mathbf{u}}^{\mathbf{i}}(x)\in\mathcal{W},p(x,\mathbf{i}%
,\mathbf{u})>\beta\right.  \right\}
\]
are open, since $\Phi_{\mathbf{u}}^{\mathbf{i}}$ and $p$ are continuous with
respect to $x$. Using Corollary \ref{Corollary_inv}(i) one finds for every
point $x\in M$ an invariant control set $C_{r},r\in\{1,\ldots,l\}$, with
$\mathrm{int}C_{r}\subset\mathcal{O}^{+}(x)$. Hence, for every $x\in M$ there
are a time $T>0$ and a control $v$ such that $\varphi(T,x,v)\in\mathcal{U}$.

Then the tube lemma, Lemma \ref{tubelemma}, implies
\[
M=\bigcup_{n\in\mathbb{N}}\left(  \bigcup_{\beta>0}\bigcup_{(\mathbf{i}%
,\mathbf{u})\in\mathbb{T}_{n}^{ij}}\mathcal{O}\left(  \mathbf{i}%
,\mathbf{u},\beta\right)  \right)  .
\]
By adding 'false' jumps to a chain $(\mathbf{i},\mathbf{u})$, we get:
\[
\forall(\mathbf{i},\mathbf{u})\in\mathbb{T}_{n}^{ij}~\forall n^{\prime}\geq
n~\forall\beta~\exists\beta^{\prime}>0~\exists(\mathbf{i^{\prime}%
},\mathbf{u^{\prime}})\in\mathbb{T}_{n^{\prime}}^{ij}:\mathcal{O}\left(
\mathbf{i},\mathbf{u},\beta\right)  \subset\mathcal{O}\left(
\mathbf{i^{\prime}},\mathbf{u^{\prime}},\beta^{\prime}\right)  .
\]
Therefore the union over $n$ is increasing, so by compactness of $M$, there is
$m_{ij}$ with
\[
M\subset\bigcup_{\beta>0}\left(  \bigcup_{(\mathbf{i},\mathbf{u})\in
\mathbb{T}_{m_{ij}}^{ij}}\mathcal{O}\left(  \mathbf{i},\mathbf{u}%
,\beta\right)  \right)  .
\]
Since there only finitely many $i$ and $j$, we can choose $m$ independently of
$i,j\in E$.

The inclusions
\[
\mathcal{O}\left(  \mathbf{i},\mathbf{u},\beta_{1}\right)  \subset
\mathcal{O}\left(  \mathbf{i},\mathbf{u},\beta_{2}\right)  \text{ for }%
\beta_{1}\geq\beta_{2}%
\]
show that the union over the $\beta$ is increasing with decreasing $\beta$.
Thus, by compactness there is $\beta_{0}>0$ such that for all $i,j\in E$
\[
M=\bigcup_{(\mathbf{i},\mathbf{u})\in\mathbb{T}_{m}^{ij}}\mathcal{O}\left(
\mathbf{i},\mathbf{u},\beta_{0}\right)  .
\]
Again compactness of $M$ shows that there is $N\in\mathbb{N}$ with
$M=\bigcup_{k=1}^{N}\mathcal{O}_{k}$, where $\mathcal{O}_{k}:=\mathcal{O}%
\left(  \mathbf{i}^{k},\mathbf{u}^{k},\beta_{0}\right)  $ for some
$(\mathbf{i}^{k},\mathbf{u}^{k})\in\mathbb{T}_{m}^{ij}$. Since $\mathrm{cl}%
\mathcal{W}_{r}\subset\mathrm{int}\mathcal{U}_{r}$ for $r=1,\ldots,l$, the
distance between $\mathrm{cl}\mathcal{W}$ and the complement of $\mathcal{U}$
is positive by compactness. So we can choose $\varepsilon$ small enough such
that for all $x\in\mathcal{O}_{k}$ and all $\mathbf{t}\in\mathbb{R}_{+}^{m}$
with $\left\Vert \mathbf{t-u}^{k}\right\Vert \leq\varepsilon$%
\[
\Phi_{\mathbf{t}}^{\mathbf{i}^{k}}(x)\in\mathcal{U}\text{ and }p(x,\mathbf{i}%
^{k},\mathbf{t})\geq\beta.
\]
Using compactness of $M$ one can here choose $\varepsilon,\beta>0$
independently of $x\in M$.
\end{proof}

We denote the $n$-step transition probability from $(x,i)$ to a measurable set
$A\subset M\times E$ by $P_{n}\left(  (x,i),A\right)  =\mathbb{E}\left[
\tilde{Z}_{n}\in A|\tilde{Z}_{0}=(x,i)\right]  $.

\begin{lemma}
\label{Lemma_inC(i)}Let $y\in\mathcal{O}^{+}(x)$ and consider a neighborhood
$\mathcal{W}(y)$ of $y$. Then for all $i,j\in E$ there are a neighborhood
$\mathcal{W}(x)$ of $x$ and $n\in\mathbb{N}$ such that for all $z\in
\mathcal{W}(x)$
\[
P_{n}\left(  (z,i),\mathcal{W}(y)\times\{j\}\right)  >0.
\]

\end{lemma}

\begin{proof}
Let $y=\varphi(T,x,v)$ for some $T>0$ and $v\in L^{\infty}(\mathbb{R}%
_{+},\mathrm{co}(S))$. The tube lemma, Lemma \ref{tubelemma}, implies that for
all $\delta>0$ there is $\varepsilon>0$ such that
\[
\mathbb{P}_{x,i}\left[  \sup_{t\in\lbrack0,T]}\Vert X_{t}-\varphi
(t,x,v)\Vert\leq\delta\right]  \geq\varepsilon.
\]
By the arguments used in the proof of Lemma \ref{lemma41} one finds for a
neighborhood $\mathcal{W}(y)$ of $y$ a natural number $n\in\mathbb{N}$ and
$\varepsilon,\beta>0$ such that for every $x\in M$ there are an open
neighborhood $\mathcal{W}(x)$ of $x$ and $(\mathbf{i},\mathbf{u})\in
\mathbb{T}_{n}^{ij}$ with the following property:

For $\mathbf{t}\in\mathbb{R}_{+}^{n}$ with $\left\Vert \mathbf{t}%
-\mathbf{u}\right\Vert \leq\varepsilon$ and $z\in\mathcal{W}(x)$ it follows
that $\Phi_{\mathbf{t}}^{\mathbf{i}}(z)\in\mathcal{W}(y)$ and $p(z,\mathbf{i}%
,\mathbf{t})\geq\beta$.

Then we have for all $z\in\mathcal{W}(x)$ and all $i,j\in E$%
\begin{align*}
P_{n}((z,i),\mathcal{W}(y)\times\{j\})  &  \geq\mathbb{P}_{(y,i)}\left(
\left\Vert \mathbf{t}-\mathbf{u}\right\Vert \leq\varepsilon,\mathbf{i}%
^{\prime}=\mathbf{i}\right) \\
&  =\mathbb{P}_{(y,i)}\left(  \left\Vert \mathbf{t}-\mathbf{u}\right\Vert
\leq\varepsilon\right)  \cdot\mathbb{P}_{(y,i)}\left(  \mathbf{i}^{\prime
}=\mathbf{i}|\left\Vert \mathbf{u}^{\prime}-\mathbf{u}\right\Vert
\leq\varepsilon\right) \\
&  \geq\mathbb{P}_{(y,i)}\left(  \left\Vert \mathbf{t}-\mathbf{u}\right\Vert
\leq\varepsilon\right)  \cdot\beta.
\end{align*}
Since the components of $\mathbf{t}$ are identically and independently
distributed and the distribution is exponential, there is $\gamma>0$ such that
$\mathbb{P}_{(y,i)}\left(  \left\Vert \mathbf{t}-\mathbf{u}\right\Vert
\leq\varepsilon\right)  \geq\gamma$, and
\[
P_{n}((z,i),\mathcal{W}(y)\times\{j\})>0.
\]

\end{proof}

Similar arguments will show the following lemma.

\begin{lemma}
\label{Lemma4.4}Let the assumptions and notation of Lemma \ref{lemma41} be
satisfied. Then there exist $m\in\mathbb{N}$ such that for all $x\in M$ and
all $i\in E$%
\[
P_{m}\left(  (x,i),\bigcup\nolimits_{r}\mathrm{int}C_{r}\times E\right)  >0.
\]

\end{lemma}

\begin{proof}
Choose the open set $\mathcal{U}\subset$ $\bigcup\nolimits_{r}\mathrm{int}%
C_{r}$ as in Lemma \ref{lemma41}. Then one finds $m\in\mathbb{N}$ and
$\varepsilon,\beta>0$ such that for every $x\in M$ and $i,j\in E$ there is an
open neighborhood $\mathcal{W}_{ij}(x)$ of $x$ and $(\mathbf{i}_{ij}%
,\mathbf{u}_{ij})\in\mathbb{T}_{n}^{ij}$ with the following property:

For $\mathbf{t}\in\mathbb{R}_{+}^{m}$ with $\left\Vert \mathbf{t}%
-\mathbf{u}_{ij}\right\Vert \leq\varepsilon$ and $y\in\mathcal{W}_{ij}(x)$ it
follows that $\Phi_{\mathbf{t}}^{\mathbf{i}_{ij}}(y)\in\bigcup_{r=1}%
^{l}\mathrm{int}C_{r}$ and $p(y,\mathbf{i}_{ij},\mathbf{t})\geq\beta$.

The set $\mathcal{W}(x):=\bigcap_{i,j\in E}\mathcal{W}_{ij}(x)$ is an open
neighborhood of $x$ and we have for all $y\in\mathcal{W}(x)$ and all $i,j\in
E$%
\begin{align*}
P_{m}((y,i),\mathcal{U}\times E)  &  \geq P_{m}((y,i),\mathcal{U}\times E)\\
&  \geq\mathbb{P}_{(y,i)}\left(  \left\Vert \mathbf{u}-\mathbf{u}%
_{ij}\right\Vert \leq\varepsilon,\mathbf{i}=\mathbf{i}_{ij}\right) \\
&  \geq\mathbb{P}_{(y,i)}\left(  \left\Vert \mathbf{u}-\mathbf{u}%
_{ij}\right\Vert \leq\varepsilon\right)  \cdot\beta>0.
\end{align*}
Since $M$ is compact, finitely many neighborhoods $\mathcal{W}(x)$ cover $M$.
\end{proof}

For an invariant measure $\mu$ of $\left(  \tilde{Z}_{n}\right)  $, a simple
calculation shows for every measurable set $A\subset M\times E$%
\begin{equation}
\mu(A)=\int_{M\times E}P_{n}\left(  (x,i),A\right)  \mu(d(x,i))\text{ for
}n\in\mathbb{N}. \label{P_n}%
\end{equation}
For brevity, we will replace $\mu(d(x,i))$ in the following by $d\mu\,$when no
confusion can occur. The following theorem shows that the supports of the
invariant measures of the discrete-time process are determined on the
invariant control sets.

\begin{theorem}
\label{theorem44}Let $M$ be a compact positively invariant set for system
(\ref{CS}) with controls in $L^{\infty}(\mathbb{R}_{+},\mathrm{co}(S))$, and
suppose that there are only finitely many invariant control sets $C_{r}\subset
M,r=1,\dots,l$, and $C_{r}=\mathrm{cl}(\mathrm{int}C_{r})$ and $\mathrm{int}%
C_{r}\subset\mathcal{O}^{+}(x)$ for every $x\in C_{r}$ and all $r$ (this holds
in particular, if $M$ is also locally accessible). Then for every invariant
measure $\mu$ of the discrete-time process $(\tilde{Z}_{n})_{n}=(\tilde{X}%
_{n},\tilde{Y}_{n})_{n}$%
\[
\mathrm{supp}\mu\subset\bigcup_{r=1}^{l}C_{r}\times E.
\]

\end{theorem}

\begin{proof}
Suppose, by way of contradiction, that there is $(x_{0},i_{0})\in
A:=\mathrm{supp}\mu\setminus\left(  \bigcup_{r}C_{r}\times E\right)  $. Since
the invariant control sets $C_{r}$ are closed, there is an open neighborhood
$\mathcal{W}(x_{0})$ of $x_{0}$ with $\left(  \mathcal{W}(x_{0})\times
E\right)  \cap\left(  \bigcup_{r}C_{r}\times E\right)  =\emptyset$. Then
\[
\mu\left(  \left(  M\setminus\bigcup\nolimits_{r}C_{r}\right)  \times
E\right)  \geq\mu(\mathcal{W}(x_{0})\times E)>0
\]
and%
\begin{align*}
1  &  =\mu(M\times E)=\mu\left(  \left(  M\setminus\bigcup\nolimits_{r}%
C_{r}\right)  \times E\right)  +\mu\left(  \bigcup\nolimits_{r}C_{r}\times
E\right)  =\mu\left(  \mathrm{supp}\mu\right) \\
&  =\mu\left(  \mathrm{supp}\mu\setminus\left(  \bigcup\nolimits_{r}%
C_{r}\times E\right)  \right)  +\mu\left(  \bigcup\nolimits_{r}C_{r}\times
E\right)  .
\end{align*}
It follows that%
\[
0<\mu\left(  \left(  M\setminus\bigcup\nolimits_{r}C_{r}\right)  \times
E\right)  =\mu\left(  \mathrm{supp}\mu\setminus\left(  \bigcup\nolimits_{r}%
C_{r}\times E\right)  \right)  =\mu(A).
\]
Lemma \ref{Lemma4.4} shows that%
\[
P_{m}\left(  (x,i),\bigcup\nolimits_{r}\mathrm{int}C_{r}\times E\right)
>0\text{ for all }(x,i)\in A,
\]
which implies%
\begin{equation}
\int_{A}P_{m}\left(  (x,i),\bigcup\nolimits_{r}\mathrm{int}C_{r}\times
E\right)  \mu(d(x,i))>0. \label{a}%
\end{equation}
By invariance of $\mu$ and positive invariance of the invariant control sets
$C_{r}$ we find%
\begin{align}
\mu(A)  &  =\int\limits_{\mathrm{supp}\mu}P_{m+1}\left(  (x,i),A\right)
d\mu\nonumber\\
&  =\underbrace{\int\limits_{\mathrm{supp}\mu\cap\left(  \bigcup C_{r}\times
E\right)  }P_{m+1}\left(  (x,i),A\right)  d\mu}_{=0}+\int
\limits_{\mathrm{supp}\mu\setminus\left(  \bigcup C_{r}\times E\right)
}P_{m+1}\left(  (x,i),A\right)  d\mu\nonumber\\
&  =\int\limits_{A}P_{m+1}\left(  (x,i),A\right)  d\mu. \label{b}%
\end{align}
Using the Chapman-Kolmogorov equation and again positive invariance of the
invariant control sets $C_{r}$, we can estimate for all $(x,i)\in A$
\begin{align*}
P_{m+1}\left(  (x,i),A\right)   &  =\int\limits_{M\times E}P_{m}\left(
(x,i),(y,j)\right)  P_{1}\left(  (y,j),A\right)  \mu(d(y,j))\\
&  =\int\limits_{\left(  M\times E\right)  \setminus\left(  \bigcup
C_{r}\times E\right)  }P_{m}\left(  (x,i),(y,j)\right)  \underbrace
{P_{1}\left(  (y,j),A\right)  }_{\leq1}\mu(d(y,j))\\
&  \qquad+\underbrace{\int\limits_{\bigcup C_{r}\times E}P_{m}\left(
(x,i),(y,j)\right)  P_{1}\left(  (y,j),A\right)  \mu(d(y,j))}_{=0}\\
&  \leq\int\limits_{\left(  M\times E\right)  \setminus\left(  \bigcup
C_{r}\times E\right)  }P_{m}\left(  (x,i),(y,j)\right)  \mu(d(y,j))\\
&  =P_{m}\left(  (x,i),(M\times E)\setminus\left(  \bigcup C_{r}\times
E\right)  \right) \\
&  =P_{m}\left(  (x,i),(M\times E)\right)  -P_{m}\left(  (x,i),\bigcup
\nolimits_{r}\mathrm{int}C_{r}\times E\right) \\
&  =1-P_{m}\left(  (x,i),\bigcup\nolimits_{r}\mathrm{int}C_{r}\times E\right)
\end{align*}
Together with (\ref{a}) and (\ref{b}) this yields the contradiction%
\begin{align*}
\mu(A)  &  =\int\limits_{A}P_{m+1}\left(  (x,i),A\right)  d\mu\\
&  \leq\mu(A)-\int_{A}P_{m}\left(  (x,i),\bigcup\nolimits_{r}\mathrm{int}%
C_{r}\times E\right)  d\mu<\mu(A).
\end{align*}

\end{proof}

We note the following property.

\begin{proposition}
\label{proposition45}Let $\mu$ be an invariant measure for the discrete-time
process $(\tilde{Z}_{n})_{n}$. Let $C$ be a compact invariant control set and
assume that local accessibility holds on $C$. If $\mathrm{supp}\mu\cap(C\times
E)\not =\emptyset$, then $C\times E\subset\mathrm{supp}\mu$ and, in
particular, $\mu(C\times E)>0$.
\end{proposition}

\begin{proof}
Suppose, contrary to the assertion, that there is $(y,j)\in(C\times
E)\setminus\mathrm{supp}\mu$. By Theorem \ref{basic1}(i) $C=\mathrm{cl}%
(\mathrm{int}C)$ and hence we may assume that $y\in\mathrm{int}C$. Thus there
is an open neighborhood $\mathcal{W}(y)\subset C$ of $y$ with
\begin{equation}
\left(  \mathcal{W}(y)\times\{j\}\right)  \cap\mathrm{supp}\mu=\emptyset.
\label{a0}%
\end{equation}
Pick $(x_{0},i_{0})\in\mathrm{supp}\mu\cap(C\times E)$. By Theorem
\ref{basic1}(i) $\mathrm{int}C\subset\mathcal{O}^{+}(x_{0})$, hence there
exists $y_{0}\in\mathcal{O}^{+}(x_{0})$ with $(y_{0},j)\in\mathcal{W}%
(y)\times\{j\}$. By Lemma \ref{Lemma_inC(i)}, we find an open neighborhood
$\mathcal{W}(x_{0})$ of $x_{0}$ such that for all $z\in\mathcal{W}(x_{0})$ and
$i\in E$%
\begin{equation}
P_{n}\left(  (z,i),\mathcal{W}(y)\times\{j\}\right)  >0\text{ and, clearly,
}\mu(\mathcal{W}(x_{0})\times E)>0. \label{a1}%
\end{equation}
This implies%
\[
\int_{M\times E}P_{n}\left(  (z,i),\mathcal{W}(y)\times\{j\}\right)  d\mu
\geq\int_{\mathcal{W}(x_{0})\times E}P_{n}\left(  (z,i),\mathcal{W}%
(y)\times\{j\}\right)  d\mu>0,
\]
and hence%
\begin{align*}
1  &  =\int_{M\times E}P_{n}\left(  (z,i),M\times E\right)  d\mu\\
&  =\int_{M\times E}P_{n}\left(  (z,i),\left(  M\times E\right)
\setminus\left(  \mathcal{W}(y)\times\{j\}\right)  \right)  d\mu+\int_{M\times
E}P_{n}\left(  (z,i),\mathcal{W}(y)\times\{j\}\right)  d\mu\\
&  >\int_{M\times E}P_{n}\left(  (z,i),\left(  M\times E\right)
\setminus\left(  \mathcal{W}(y)\times\{j\}\right)  \right)  d\mu.
\end{align*}
Using also (\ref{P_n}) and (\ref{a0}), we obtain the contradiction%
\begin{align*}
1  &  =\mu(\mathrm{supp}\mu)=\int_{M\times E}P_{n}\left(  (z,i),\mathrm{supp}%
\mu\right)  d\mu\\
&  \leq\int_{M\times E}P_{n}\left(  (z,i),\left(  M\times E\right)
\setminus\left(  \mathcal{W}(y)\times\{j\}\right)  \right)  d\mu<1.
\end{align*}

\end{proof}

Next we discuss the ergodic case. Recall that an invariant measure $\mu$ is
ergodic (extremal) if it cannot be written as a proper convex combination of
invariant measures.

\begin{theorem}
\label{Theorem47}(i) Assume that system (\ref{CS}) with controls in
$L^{\infty}(\mathbb{R}_{+},\mathrm{co}(S))$ is locally accessible on a compact
positively invariant set $M$. Then for every ergodic measure $\mu$ of the
discrete-time process $(\tilde{Z}_{n})_{n}$ there is a compact invariant
control set $C$ with $\mathrm{supp}\mu=C\times E$.

(ii) Conversely, let $C$ be a compact invariant control set. Then there exists
an ergodic measure with support equal to $C\times E$ and every invariant
measure with support contained in $C\times E$ has support equal to $C\times E$.

(iii) Assume that for some $x$ in a compact invariant control set $C$ the Lie
algebra $\mathcal{LA}(F^{0},\ldots,F^{m})$ has full rank at $x.$ Then there is
a unique invariant measure $\mu$ supported by $C\times E$ (hence $\mu$ is
ergodic) and there are nonnegative constants $c$ and $\rho$ with $\rho<1$ such
that for all $(x,i)\in C\times E$ and Borel sets $A\subset C$%
\begin{equation}
|\mathbb{P}_{x,i}[\tilde{Z}_{n}\in A]-\mu(A)|\leq c\rho^{n},n\in\mathbb{N}.
\label{a4}%
\end{equation}

\end{theorem}

\begin{proof}
(i) Theorem \ref{theorem44} shows that $\mathrm{supp}\mu\subset\bigcup
\nolimits_{r}C_{r}\times E$ and it remains to prove the converse inclusion. In
view of Proposition \ref{proposition45}, we have to show that the support of
$\mu$ can intersect only one set of the form $C_{r}\times E$ for an invariant
control set $C_{r}$. Let $(x,i)\in\mathrm{supp}\mu$ for some $x$ in the
interior of $C_{r}$ and some $i\in E$. Then $\mu(C_{r}\times E)>0$ and for
$A\subset C_{r}\times E$%
\[
\mu(A)=\int_{C_{r}\times E}P((x,i),A)d\mu.
\]
Here it suffices to integrate over $C_{r}\times E$, since for the other points
$(x,i)$ in the support of $\mu$ one has that $x\in C_{s},s\not =r$,\ implying
that the probability to reach $C_{r}$ vanishes. Hence the conditional
probability measure induced by $\mu$ on $C_{r}\times E$ is invariant. If
\[
\mu\left(  \bigcup\nolimits_{s\not =r}C_{s}\times E\right)  >0,
\]
then $\mu$ is not ergodic, since it can be written as a proper convex
combination of the conditional probability measures induced by $\mu$ on
$C_{r}\times E$ and on $\bigcup\nolimits_{s\not =r}C_{s}\times E$, respectively

(ii) Existence follows from Feller continuity, compactness and positive
invariance of $C\times E.$ The second statement follows from Proposition
\ref{proposition45}.

(iii) Uniqueness and the exponential estimate follow from Bena\"{\i}m, Le
Borgne, Malrieu and Zitt \cite[Theorem 4.5]{BBMZ15}. In fact, this theorem
considers for a compact positively invariant set $M$ the set $\Gamma$ defined
in (\ref{gamma}) and assumes that there is $x\in\Gamma$ such that the Lie
algebra $\mathcal{LA}(F^{0},\ldots,F^{m})$ has full rank at $x$. Then it
concludes that (\ref{a4}) holds for all Borel sets $A\subset\Gamma$. Here we
choose $M=C$. Then Proposition \ref{proposition_gamma} implies $C=\Gamma$ and
the assertion follows.
\end{proof}

Finally, we note the following consequence for the continuous-time process.

\begin{theorem}
(i) The assertions of Theorem \ref{Theorem47}(i) and (ii) also hold for the
continuous-time process $(Z_{t})_{t\geq0}$ defined in (\ref{Z}). In
particular, if for some $x$ in a compact invariant control set $C$ the Lie
algebra $\mathcal{LA}(F^{0},\ldots,F^{m})$ has full rank at $x$, there is a
unique invariant measure $\mu$ supported by $C\times E$ (hence $\mu$ is ergodic).

(ii) If instead of the rank condition for the Lie algebra $\mathcal{LA}%
(F^{0},\ldots,F^{m})$ the rank of the smallest Lie algebra containing all
\textquotedblleft control vector fields\textquotedblright\ $F^{i}%
-F^{0},i\not =0$, in (\ref{CS2}) and all Lie brackets with $F^{i}%
,i=0,\ldots,n$, is considered, convergence of the distributions follows, hence
there are constants $c>1$ and $\alpha>0$ such that for all $(x,i)\in C\times
E$ and Borel sets $A\subset C\times E$%
\begin{equation}
|\mathbb{P}_{x,i}[Z_{t}\in A]-\mu(A)|\leq ce^{-\alpha t},t\geq0. \label{a5}%
\end{equation}

\end{theorem}

\begin{proof}
(i) The analogues of assertions (i) and (ii) in Theorem \ref{Theorem47} hold,
since by \cite[Proposition 2.4]{BBMZ15} there is a homeomorphism between the
invariant measures of the discrete-time process and the continuous-time
process mapping ergodic measures onto ergodic measures; by \cite[Lemma
2.6]{BBMZ15} the homeomorphism preserves the supports.

(ii) The exponential convergence (\ref{a5}) follows by \cite[Theorem
4.4]{BBMZ15}.
\end{proof}

\textbf{Acknowledgements}. We are grateful to an anonymous reviewer for the
careful reading which helped to eliminate a number of inconsistencies and
errors. This work was supported by the SNF grants 200020 149871 and 200021
163072 and DFG grant Co 124/19-1.

\end{document}